\pdfoutput=1
\RequirePackage{ifpdf}
\ifpdf % We are running pdfTeX in pdf mode
\documentclass[pdftex]{sigma}
\else
\documentclass{sigma}
\fi

\numberwithin{equation}{section}

\newtheorem{thm}{Theorem}[section]

\newtheorem{lem}[thm]{Lemma}
\newtheorem{prop}[thm]{Proposition}

 { \theoremstyle{definition}
\newtheorem{defin}[thm]{Definition}
}

\DeclareMathOperator{\spann}{span}
\DeclareMathOperator{\domain}{Dom}
\DeclareMathOperator{\sign}{sign}
\DeclareMathOperator*{\id}{Id}

\begin{document}

\allowdisplaybreaks

\renewcommand{\thefootnote}{$\star$}

\newcommand{\arXivNumber}{1505.01653}

\renewcommand{\PaperNumber}{073}

\FirstPageHeading

\ShortArticleName{Potential and Sobolev Spaces Related to Symmetrized Jacobi Expansions}

\ArticleName{Potential and Sobolev Spaces\\ Related to Symmetrized Jacobi Expansions\footnote{This paper is a~contribution to the Special Issue
on Orthogonal Polynomials, Special Functions and Applications.
The full collection is available at \href{http://www.emis.de/journals/SIGMA/OPSFA2015.html}{http://www.emis.de/journals/SIGMA/OPSFA2015.html}}}

\Author{Bartosz LANGOWSKI}

\AuthorNameForHeading{B.~Langowski}

\Address{Wydzia\l{} Matematyki,
      Politechnika Wroc\l{}awska,\\
      Wyb{.}\ Wyspia\'nskiego 27,
      50--370 Wroc\l{}aw, Poland}
\Email{\href{mailto:bartosz.langowski@pwr.edu.pl}{bartosz.langowski@pwr.edu.pl}}

\ArticleDates{Received May 08, 2015, in f\/inal form September 10, 2015; Published online September 12, 2015}

\Abstract{We apply a symmetrization procedure to the setting of Jacobi expansions and study potential spaces
in the resulting situation. We prove that the potential spaces of integer orders
are isomorphic to suitably def\/ined Sobolev spaces. Among further results, we obtain
a fractional square function characterization, structural theorems and Sobolev type
embedding theorems for these potential spaces.}


\Classification{42C10; 42C05; 42C20}

\renewcommand{\thefootnote}{\arabic{footnote}}
\setcounter{footnote}{0}

\section{Introduction} \label{sec:intro}

This article is motivated by the recent results of Nowak and Stempak \cite{Symmetrized} and the
author's papers \cite{L2,L3}. In \cite{L2} we investigated Sobolev and potential spaces related to
discrete Jacobi function expansions. The main achievement of \cite{L2} is a suitable def\/inition of
the Sobolev--Jacobi spaces so that they are isomorphic with the potential spaces with appropriately
chosen parameters. The article~\cite{L2} is a continuation and extension of a similar study conducted
in the setting of ultraspherical expansions by Betancor et al.~\cite{betancor}.
The other author's paper~\cite{L3} contains further investigations of the Jacobi potential spaces.
The most important outcome of~\cite{L3} is a~characterization of the potential spaces by means of
suitably def\/ined fractional square functions. The research in~\cite{L3} was inspired by another paper of
Betancor et al.~\cite{betsq}, in which a general technique of using square functions in analysis of
potential spaces associated with discrete and continuous orthogonal expansions was developed.

On the other hand, in \cite{Symmetrized} Nowak and Stempak proposed a symmetrization procedure in a~context of general discrete orthogonal expansions related to a second-order dif\/ferential ope\-ra\-tor~$L$,
a `Laplacian'. This procedure, combined with a unif\/ied conjugacy scheme established in an earlier article
by the same authors~\cite{NoSt}, allows one to associate, via a suitable embedding, a~dif\/ferential-dif\/ference
`Laplacian'~$\mathbb{L}$ with the initially given orthogonal system of eigenfunctions of~$L$ so that the
resulting extended conjugacy scheme has the natural classical shape. In particular, the related `partial
derivatives' decomposing~$\mathbb{L}$ are skew-symmetric in an appropriate~$L^2$ space and they commute with
Riesz transforms and conjugate Poisson integrals. Thus the symmetrization procedure overcomes the main
inconvenience of the theory postulated in~\cite{NoSt}, that is the lack of symmetry in the principal objects
and relations resulting in essential deviations of the theory from the classical shape.  The price is,
however, that the `Laplacian'~$\mathbb{L}$ and the associated `derivatives' are not dif\/ferential, but
dif\/ferential-dif\/ference operators.
It was shown in~\cite{Symmetrized} that the symmetrization is supported by a good~$L^2$ theory.
Moreover, in~\cite{L1} the author verif\/ied that further support comes from the~$L^p$ theory, at least
when the Jacobi polynomial context is considered.

In the present paper we apply the above mentioned symmetrization procedure to the setting of Jacobi
function expansions considered in~\cite{L2,L3}. Then we def\/ine and study the associated potential
spaces and Sobolev spaces. As the main results, we establish an isomorphism between these spaces
(Theorem~\ref{thm:sob1}) and characterize the potential spaces by means of suitably def\/ined fractional
square functions (Theorems~\ref{thm:char} and~\ref{equivfun'}). Among further results, we prove
structural and embedding theorems for the potential spaces, in particular we obtain
a counterpart of the classical Sobolev embedding theorem in the Jacobi setting (see Theorems~\ref{thm:struc}
and~\ref{thm:embed}). All of this extends the results from~\cite{L2,L3} to the symmetrized situation.

The general strategy we use to prove the results in the symmetrized setting relies on two steps.
In the f\/irst step we exploit symmetries of the operators under consideration in order to reduce the
analysis essentially to the initial non-symmetrized case. Then the second step consists in taking
advantage of the results already existing in the literature, mostly from author's previous
articles \cite{L2,L3}. Even though the general line of reasoning is relatively easy, some details
occur to be rather technical and complex.

An important aspect, and in fact also a partial motivation of our research, is the suggestion from
\cite[Section 5]{Symmetrized} that the symmetrization could have a signif\/icant impact on developing
the theory of Sobolev spaces related to orthogonal expansions. This concerns, in particular, higher-order
`derivatives' leading to appropriate Sobolev spaces. It turns out, however, that in our symmetrized
framework the relevant higher-order `derivatives' are not constructed from the f\/irst-order `derivative'
(see Proposition~\ref{prop:neg}), as one would perhaps expect after reading the optimistic comments
in \cite[Section~5]{NoSt}.
Thus these derivatives are even more exotic than the variable index derivatives that are suitable in
the initial non-symmetrized Jacobi setting. So it seems that the symmetrization brings no improvement
in dealing with Sobolev spaces, at least in the Jacobi setting considered. This makes a noteworthy
contrast to the conjugacy scheme which benef\/its a lot from the symmetrization.

Sobolev and potential spaces related to dif\/ferent classical orthogonal expansions were investigated
in recent years by various authors, see, e.g.,
\cite{BaUr1,BaUr2,betancor,betsq,BT1,BT2,Graczyk,L2,L3,RT}.
On the other hand, harmonic analysis in several frameworks of Jacobi expansions was in the last decade
studied in \cite{BaUr1,BaUr2,CU,CNS,L1,L2,L3,Nowak&Roncal,NoSj,NS1,NS2,parameters,Stempak} (see also
\cite{CRS,CRS2}), among others. Our present work contributes to both of these lines of research.

An interesting study of variable exponent Sobolev spaces for Jacobi expansions is contained in the recent
paper by Almeida, Betancor, Castro, Sanabria and Scotto~\cite{varex}. The results of~\cite{varex} are
related to those in the author's papers~\cite{L2,L3}, but were obtained independently. In particular,
there is a partial overlap in characterizations of the Jacobi potential spaces via fractional square
functions obtained in~\cite{L3} and~\cite{varex}. We thank one of the referees for bringing~\cite{varex}
to our attention.

\subsection*{Notation}

Throughout the paper, we use a fairly standard notation with essentially all symbols referring either to
the measure space $((-\pi,\pi),d\theta)$ or the restricted space $((0,\pi),d\theta)$.
Given a function~$f$ on $(-\pi,\pi)$, we denote by $f^{+}$ its restriction to the subinterval $(0,\pi)$,
and by $f_{\textrm{even}}$ and $f_{\textrm{odd}}$ its even and odd parts, respectively,
\begin{gather*}
f_{\textrm{even}}(\theta) = \frac{f(\theta)+f(-\theta)}2, \qquad
f_{\textrm{odd}}(\theta) = \frac{f(\theta)-f(-\theta)}2.
\end{gather*}
We let
\begin{gather*}
\langle f_1,f_2 \rangle = \int_{-\pi}^{\pi} f_1(\theta)\overline{f_2(\theta)}\, d\theta, \qquad
\langle h_1,h_2 \rangle_{+} = \int_0^{\pi} h_1(\theta)\overline{h_2(\theta)} \, d\theta,
\end{gather*}
whenever the integrals make sense. For $1 \le p \le \infty$, $p'$~denotes its conjugate exponent,
$1/p+1/p'=1$. When writing estimates, we will use the notation $X \lesssim Y$ to indicate that
$X \le CY$ with a positive constant $C$ independent of signif\/icant quantities. We shall write
$X \simeq Y$ when simultaneously $X \lesssim Y$ and $Y \lesssim X$.

\section{Preliminaries}\label{sec:prel}

Given parameters $\alpha, \beta>-1$, consider the Jacobi dif\/ferential operator
\begin{gather*}
L_{\alpha,\beta}= -\frac{d^2}{d\theta^2}
-\frac{1-4\alpha^2}{16\sin^2\frac{\theta}{2}}-\frac{1-4\beta^2}{16\cos^2\frac{\theta}{2}}
 = D_{\alpha,\beta}^* D_{\alpha,\beta} + A^2_{\alpha,\beta};
\end{gather*}
here $A_{\alpha,\beta}=(\alpha+\beta+1)/2$ is a constant, and
\begin{gather} \label{der_ini}
D_{\alpha,\beta}=\frac{d}{d\theta}-\frac{2\alpha+1}{4}\cot\frac{\theta}{2}+\frac{2\beta+1}{4}\tan\frac{\theta}{2},
\qquad D_{\alpha,\beta}^* = D_{\alpha,\beta}-2\frac{d}{d\theta},
\end{gather}
are the f\/irst-order `derivative' naturally associated with $L_{\alpha,\beta}$ and
its formal adjoint in $L^2(0,\pi)$, respectively.
It is well known that $L_{\alpha,\beta}$, def\/ined initially on $C_c^2(0,\pi)$, has a non-negative self-adjoint
extension in $L^2(0,\pi)$ whose spectral decomposition is discrete and given by the Jacobi functions
$\phi_n^{\alpha,\beta}$, $n \ge 0$. The corresponding eigenvalues are $\lambda_n^{\alpha,\beta} = (n+A_{\alpha,\beta})^2$,
and the system $\{\phi_n^{\alpha,\beta}\colon n\ge 0\}$ constitutes an orthonormal basis in $L^2(0,\pi)$.
Some problems in harmonic analysis related to~$L_{\alpha,\beta}$ were investigated recently in~\cite{varex,L2,L3,Nowak&Roncal,NS2,Stempak}.

When ${\alpha,\beta} \ge -1/2$, the functions $\phi_n^{\alpha,\beta}$ belong to all $L^p(0,\pi)$, $1\le p \le \infty$.
However, if $\alpha< -1/2$ or $\beta < -1/2$, then $\phi_n^{\alpha,\beta}$ are in $L^p(0,\pi)$ if and only if
$p < -1/\min(\alpha+1/2,\beta+1/2)$. This leads to the restriction
$p'({\alpha,\beta}) < p < p({\alpha,\beta})$ for $L^p$ mapping properties of various operators associated with
$L_{\alpha,\beta}$, where
\begin{gather*}
p({\alpha,\beta}):= 	\begin{cases}
						\infty, & {\alpha,\beta} \ge -1/2, \\
						-1/\min(\alpha+1/2,\beta+1/2), & \textrm{otherwise}.
					\end{cases}
\end{gather*}
Consequently, many results in harmonic analysis of $L_{\alpha,\beta}$ are restricted to $p \in E({\alpha,\beta})$,
\begin{gather*}
E({\alpha,\beta}) := \big(p'({\alpha,\beta}),p({\alpha,\beta})\big).
\end{gather*}

In this work we shall consider the setting related to the larger interval $(-\pi,\pi)$ equipped with
Lebesgue measure. An application of the symmetrization procedure from \cite{Symmetrized} to the
context of~$L_{\alpha,\beta}$ brings in the following symmetrized Jacobi `Laplacian' and the associated `derivative':
\begin{gather*}
\mathbb{L}_{\alpha,\beta}=-\mathbb{D}_{\alpha,\beta}^2+A_{\alpha,\beta}^2,
\end{gather*}
with
\begin{gather*}
\mathbb{D}_{\alpha,\beta}f=\frac{df}{d\theta}-\left(\frac{2\alpha+1}{4}
\cot\frac{\theta}{2}-\frac{2\beta+1}{4}\tan\frac{\theta}{2}\right)\check{f}=D_{\alpha,\beta}f_{\textrm{even}}-D_{\alpha,\beta}^*f_{\textrm{odd}},
\end{gather*}
where $\check{f}(\theta)=f(-\theta)$ is the ref\/lection of $f$, and $D_{\alpha,\beta}$ and $D_{\alpha,\beta}^*$ are
given on $(-\pi,\pi)$ by~\eqref{der_ini}.
Note that, due to the ref\/lection occurring in~$\mathbb{D}_{\alpha,\beta}$, we deal here with a Dunkl type operator.
For more details concerning Jacobi--Dunkl operators we refer to~\cite{chouchene}, see also
\cite[Section~7]{BK}.

Also, the following remark is in order. Formally, the space underlying the symmetrized setting
is the sum $(-\pi,0) \cup (0,\pi)$. Nevertheless, often it can (and will) be identif\/ied with the
interval $(-\pi,\pi)$, since for some aspects of the theory the single point $\theta=0$ is negligible.
A~typical example here are~$L^p$ inequalities which ``do not see'' sets of null measure. On the other hand,
some objects in the symmetrized situation may not even be properly def\/ined at $\theta=0$ (the latter
may in addition depend on the parameters of type), hence this point must be excluded from some considerations
like, for instance, continuity or smoothness questions. That is why in what follows several times
$(-\pi,\pi){\setminus} \{0\}$ appears rather than~$(-\pi,\pi)$.

The orthonormal basis in $L^2(-\pi,\pi)$ arising from the symmetrization procedure applied
to the system of Jacobi functions is
\begin{gather*} %\label{Phiphi}
\Phi_n^{\alpha,\beta}=\frac{1}{\sqrt{2}}\begin{cases}\phi_{n/2}^{\alpha,\beta},&  \textrm{$n$ even},\\
-\big(\lambda_{(n+1)/2}^{\alpha,\beta}-\lambda_0^{\alpha,\beta}\big)^{-1/2} D_{\alpha,\beta}\phi_{(n+1)/2}^{\alpha,\beta}, &
  \textrm{$n$ odd},
\end{cases}
\end{gather*}
where $\phi_n^{\alpha,\beta}$ are even extensions (denoted still by the same symbol)
to $(-\pi, \pi)$ of the Jacobi functions. More precisely,
\begin{gather*}
\phi_n^{\alpha,\beta}(\theta) = c_n^{\alpha,\beta}  \Psi^{\alpha,\beta}(\theta) P_n^{\alpha,\beta}(\cos\theta),
	\qquad \theta \in (-\pi,\pi), \quad n \ge 0,
\end{gather*}
where $c_n^{\alpha,\beta}$ are suitable normalizing constants, $P_n^{\alpha,\beta}$ are the classical Jacobi polynomials
as def\/ined in Szeg\H{o}'s monograph~\cite{Sz}, and
\begin{gather*}
\Psi^{\alpha,\beta}(\theta) := \left|\sin\frac{\theta}2\right|^{\alpha+1/2}\left(\cos\frac{\theta}2\right)^{\beta+1/2}.
\end{gather*}
Observe that $D_{\alpha,\beta}f$ is an odd (even) function if~$f$ is even (odd). Consequently,
$\Phi_n^{\alpha,\beta}$ is even (odd) if and only if~$n$ is an even (odd) number. By using \cite[formula~(5)]{L2}
we f\/ind that
\begin{gather} \label{Phi}
\Phi_{2n}^{\alpha,\beta}(\theta) = \frac{1}{\sqrt{2}} \phi_n^{\alpha,\beta}(\theta), \qquad
\Phi_{2n+1}^{\alpha,\beta}(\theta) = \frac{1}{\sqrt{2}} \sign(\theta)\phi_{n}^{\alpha+1,\beta+1}(\theta), \qquad n \ge 0.
\end{gather}
Notice that when $\alpha \ge -1/2$, all $\Phi_n^{\alpha,\beta}$, $n \ge 0$, are continuous functions on $(-\pi,\pi)$;
on the other hand, for $\alpha < -1/2$ and $n$ even a singularity at $\theta = 0$ occurs.
It is a nice coincidence that in our setting~$\Phi_n^{\alpha,\beta}$ are essentially~$\phi_k^{\alpha,\beta}$ or~$\phi_k^{\alpha,\beta}$ with shifted parameters. Roughly speaking, this makes the analysis in the symmetrized
situation reducible to the analysis in the initial, non-symmetrized setting. In general, and even in other
Jacobi contexts (see, e.g.,~\cite{L1}), things are more complicated.

According to \cite[Lemma~3.5]{Symmetrized}, each $\Phi_{n}^{\alpha,\beta}$ is an eigenfunction of the
symmetrized Jacobi operator. More precisely,
\begin{gather*}
\mathbb{L}_{\alpha,\beta}\Phi_n^{\alpha,\beta} = \lambda^{\alpha,\beta}_{\langle n \rangle} \Phi_n^{\alpha,\beta},
	\qquad n \ge 0,
\end{gather*}
where we use the notation $\langle n \rangle=\big\lfloor\frac{n+1}2 \big\rfloor$ introduced in
\cite{Symmetrized} (here $\lfloor\cdot\rfloor$ denotes the f\/loor function). Thus $\mathbb{L}_{\alpha,\beta}$,
considered initially on $C_c^{2}((-\pi,\pi){\setminus} \{0\})$, has a natural self-adjoint extension to
$L^2(-\pi,\pi)$, denoted by the same symbol, and given by
\begin{gather*}
\mathbb{L}_{\alpha,\beta} f = \sum_{n=0}^{\infty} \lambda_{\langle n \rangle}^{\alpha,\beta}
	\langle f, \Phi_n^{\alpha,\beta} \rangle \Phi_n^{\alpha,\beta}
\end{gather*}
on the domain $\domain\mathbb{L}_{\alpha,\beta}$ consisting of all functions $f\in L^2(-\pi, \pi)$
for which the def\/ining series converges in $L^2(-\pi,\pi)$; see \cite[Section~4]{Symmetrized}.

Next, we gather some facts about potential operators associated with $\mathbb{L}_{\alpha,\beta}$. Let $\sigma > 0$.
We consider the {Riesz type potentials} $\mathbb{L}_{\alpha,\beta}^{-\sigma}$ assuming that $\alpha+\beta\neq -1$
(when $\alpha+\beta = -1$, the bottom eigenvalue of $\mathbb{L}_{\alpha,\beta}$ is $0$) and the {Bessel type
potentials} $(\id+\mathbb{L}_{\alpha,\beta})^{-\sigma}$ with no restrictions on~$\alpha$ and~$\beta$.
Clearly, these operators are well def\/ined spectrally and bounded in~$L^2(-\pi,\pi)$.
The spectral decomposition of $\mathbb{L}_{\alpha,\beta}^{-\sigma}$ is given by
\begin{gather*}
\mathbb{L}_{\alpha,\beta}^{-\sigma}f=\sum_{n=0}^{\infty}\big(\lambda_{\langle n\rangle}^{\alpha,\beta}\big)^{-\sigma}
	\langle f,\Phi_n^{\alpha,\beta}\rangle \Phi_n^{\alpha,\beta}, \qquad f\in L^2(-\pi,\pi).
\end{gather*}
Splitting $f$ into its even and odd parts, we can write
\begin{gather} \nonumber
\mathbb{L}_{\alpha,\beta}^{-\sigma} f    = \mathbb{L}_{\alpha,\beta}^{-\sigma} f_{\textrm{even}}
	+ \mathbb{L}_{\alpha,\beta}^{-\sigma} f_{\textrm{odd}} \\ \nonumber
\hphantom{\mathbb{L}_{\alpha,\beta}^{-\sigma} f}{} = \sum_{n=0}^{\infty} \big(\lambda_{n}^{\alpha,\beta}\big)^{-\sigma} \langle f_{\textrm{even}},
	\Phi_{2n}^{\alpha,\beta} \rangle \Phi_{2n}^{\alpha,\beta} + \sum_{n=0}^{\infty} \big(\lambda_{n+1}^{\alpha,\beta}\big)^{-\sigma}
	 \langle f_{\textrm{odd}}, \Phi_{2n+1}^{\alpha,\beta} \rangle \Phi_{2n+1}^{\alpha,\beta} \\
\hphantom{\mathbb{L}_{\alpha,\beta}^{-\sigma} f}{} \equiv (\mathbb{L}_{\alpha,\beta}^{-\sigma})_{\textrm{e}}f_{\textrm{even}}
	+ (\mathbb{L}_{\alpha,\beta}^{-\sigma})_{\textrm{o}}f_{\textrm{odd}}. \label{dec}
\end{gather}
This is the decomposition of $\mathbb{L}_{\alpha,\beta}^{-\sigma}f$ into its even and odd parts, since the two terms
in \eqref{dec} are~even and odd functions, respectively. Clearly, an analogous decomposition holds for
\linebreak $(\id + \mathbb{L}_{\alpha,\beta})^{-\sigma}$. We shall use these facts in the sequel.

\begin{prop}  \label{ogr}
Let ${\alpha,\beta} > -1$ and $\sigma > 0$. Assume that $p > p'({\alpha,\beta})$ and $1 \le q < p({\alpha,\beta})$. Then
$\mathbb{L}_{\alpha,\beta}^{-\sigma}$, $\alpha+\beta \neq -1$, and $(\id+\mathbb{L}_{\alpha,\beta})^{-\sigma}$,
defined initially on $L^2(-\pi,\pi)$, extend to  bounded operators from
$L^p(-\pi,\pi)$ to $L^q(-\pi,\pi)$ if and only if
\begin{gather*}
\frac{1}q \ge \frac{1}{p} - 2\sigma.
\end{gather*}

Moreover, these operators extend to bounded operators from
$L^p(-\pi,\pi)$ to $L^{\infty}(-\pi,\pi)$ if and only if
\begin{gather*}
{\alpha,\beta}\geq-1/2\qquad\textrm{and}\qquad \frac{1}{p}<2\sigma.
\end{gather*}
\end{prop}

\begin{proof}
We consider only the Riesz type potentials since the arguments for the Bessel type potentials are parallel.
Def\/ine the restricted operators acting initially on the smaller space~$L^2(0,\pi)$:
\begin{gather*}
(\mathbb{L}_{\alpha,\beta}^{-\sigma})_{\textrm{e}}^+ h   = \sum_{n=0}^{\infty} \big(\lambda_{n}^{\alpha,\beta}\big)^{-\sigma}
	\big\langle h, \big(\Phi_{2n}^{\alpha,\beta}\big)^+ \big\rangle_{+} \big(\Phi_{2n}^{\alpha,\beta}\big)^+,\qquad h\in L^{2}(0,\pi),\\
(\mathbb{L}_{\alpha,\beta}^{-\sigma})_{\textrm{o}}^+ h   = \sum_{n=0}^{\infty} \big(\lambda_{n+1}^{\alpha,\beta}\big)^{-\sigma}
	\big\langle h, \big(\Phi_{2n+1}^{\alpha,\beta}\big)^+ \big\rangle_{+} \big(\Phi_{2n+1}^{\alpha,\beta}\big)^+,\qquad h\in L^{2}(0,\pi).
\end{gather*}
Since $\|F\|_q \simeq \|F_{\textrm{even}}\|_q + \|F_{\textrm{odd}}\|_{q}$, taking into account
\eqref{dec} we have, for $f \in L^2(-\pi,\pi)$,
\begin{gather*}
\|\mathbb{L}_{\alpha,\beta}^{-\sigma}f\|_{L^q(-\pi,\pi)}   \simeq
	\big\|(\mathbb{L}_{\alpha,\beta}^{-\sigma})_{\textrm{e}}f_{\textrm{even}}\big\|_{L^q(-\pi,\pi)}
	+ \big\|(\mathbb{L}_{\alpha,\beta}^{-\sigma})_{\textrm{o}}f_{\textrm{odd}}\big\|_{L^q(-\pi,\pi)} \\
\hphantom{\|\mathbb{L}_{\alpha,\beta}^{-\sigma}f\|_{L^q(-\pi,\pi)}}{} = 2^{1/q} \Big( \big\|(\mathbb{L}_{\alpha,\beta}^{-\sigma})_{\textrm{e}}^+f^+_{\textrm{even}}\big\|_{L^q(0,\pi)}
	+ \big\|(\mathbb{L}_{\alpha,\beta}^{-\sigma})_{\textrm{o}}^+f_{\textrm{odd}}^+\big\|_{L^q(0,\pi)} \Big).
\end{gather*}
Thus the assertion we must prove is equivalent to the following:
$(\mathbb{L}_{\alpha,\beta}^{-\sigma})_{\textrm{e}}^+$ and $(\mathbb{L}_{\alpha,\beta}^{-\sigma})_{\textrm{o}}^+$,
def\/ined initially on~$L^2(0,\pi)$, extend simultaneously to bounded operators from $L^p(0,\pi)$ to
$L^q(0,\pi)$ if and only if $\frac{1}q \ge \frac{1}p - 2\sigma$; moreover, these operators extend
simultaneously to bounded operators from~$L^p(0,\pi)$ to $L^{\infty}(0,\pi)$ if and only if ${\alpha,\beta} \ge -1/2$
and $\frac{1}p < 2\sigma$.

Now it is enough to observe that, in view of \eqref{Phi} and the identity
$\lambda_{n+1}^{\alpha,\beta} = \lambda_n^{\alpha+1,\beta+1}$, the operators
$(\mathbb{L}_{\alpha,\beta}^{-\sigma})_{\textrm{e}}^+$ and $(\mathbb{L}_{\alpha,\beta}^{-\sigma})_{\textrm{o}}^+$ coincide,
up to the constant factor $1/2$, with the Riesz type potentials $L_{\alpha,\beta}^{-\sigma}$ and
$L_{\alpha+1,\beta+1}^{-\sigma}$ related to~$L_{\alpha,\beta}$ and investigated in~\cite{Nowak&Roncal}.
The conclusion then follows by \cite[Theorem~2.4]{Nowak&Roncal}, see \cite[Proposition~2.1]{L3}.
\end{proof}

The extensions from Proposition~\ref{ogr} are unique provided that $p < \infty$.
In this case we denote them by still the same and common symbol $\mathbb{L}_{\alpha,\beta}^{-\sigma}$.
It is worth noting that all these extensions are actually realized by an integral operator with a positive
kernel. But this fact is irrelevant for our purposes, therefore we omit further details.

Denote
\begin{gather*}
\mathbb{S}_{\alpha,\beta} := \spann\big\{\Phi_n^{\alpha,\beta}\colon n \ge 0 \big\}.
\end{gather*}
Since $\{\Phi_n^{\alpha,\beta}\colon n \ge 0\}$ is an orthonormal basis, $\mathbb{S}_{\alpha,\beta}$ is dense in $L^2(-\pi,\pi)$.
The latter property remains true in some $L^p$ spaces.
\begin{lem} \label{lem:dens}
Let ${\alpha,\beta}>-1$ and $1 \le p < p({\alpha,\beta})$. Then $\mathbb{S}_{\alpha,\beta}$ is a dense subspace of $L^p(-\pi,\pi)$.
\end{lem}

\begin{proof}
Take $f \in L^p(-\pi,\pi)$.
It suf\/f\/ices to approximate separately $f_{\textrm{even}}$ and $f_{\textrm{odd}}$.
Recall that the systems $\{\Phi_{2n}^{\alpha,\beta}\colon n \ge 0\}$ and $\{\Phi_{2n+1}^{\alpha,\beta}\colon n \ge 0\}$ consist of
even and odd functions, respectively. Moreover, each of these systems when restricted to $(0,\pi)$
is linearly dense in~$L^p(0,\pi)$, see~\eqref{Phi} and \cite[Lemma~2.3]{Stempak}.
Consequently, one can approximate~$f_{\textrm{even}}$ and $f_{\textrm{odd}}$ in~$ L^p(-\pi,\pi)$ by
f\/inite linear combinations of $\Phi_{2n}$, $n \ge 0$, and $\Phi_{2n+1}$, $n \ge 0$, respectively.
\end{proof}

\begin{lem} \label{lem:wsp}
Let ${\alpha,\beta}>-1$,  $p\in E({\alpha,\beta})$ and assume that $f \in L^p(-\pi,\pi)$.
If $\langle f, \Phi_n^{\alpha,\beta}\rangle =0$ for all $n \ge 0$, then $f \equiv 0$.
\end{lem}

\begin{proof}
It is enough to observe that the lemma holds for $f \in \mathbb{S}_{\alpha,\beta}$ and then use the density of
$\mathbb{S}_{\alpha,\beta}$ (see Lemma~\ref{lem:dens}) in the dual space $(L^p(-\pi,\pi))^* = L^{p'}(-\pi,\pi)$.
\end{proof}

\begin{prop}\label{inj}
Let ${\alpha,\beta}>-1$ and $p\in E({\alpha,\beta})$. For each $\sigma>0$, $\mathbb{L}_{\alpha,\beta}^{-\sigma}$, $\alpha+\beta \neq -1$,
and $(\id + \mathbb{L}_{\alpha,\beta})^{-\sigma}$ are injective on $L^p(-\pi,\pi)$.
\end{prop}

\begin{proof}
We focus ourselves on $\mathbb{L}_{\alpha,\beta}^{-\sigma}$ and
essentially repeat the arguments from the proof of \cite[Proposition~1]{betancor}.
Notice that for $f \in \mathbb{S}_{\alpha,\beta}$
\begin{gather} \label{id}
\langle\mathbb{L}_{\alpha,\beta}^{-\sigma}f, \Phi_n^{\alpha,\beta}\rangle =
	\big(\lambda_{\langle n\rangle}^{\alpha,\beta}\big)^{-\sigma} \langle f, \Phi_n^{\alpha,\beta}\rangle, \qquad n \ge 0.
\end{gather}
By H\"older's inequality and the $L^p$-boundedness of $\mathbb{L}_{\alpha,\beta}^{-\sigma}$
(see Proposition \ref{ogr}), the functionals
\begin{gather*}
f \mapsto \langle\mathbb{L}_{\alpha,\beta}^{-\sigma}f, \Phi_n^{\alpha,\beta}\rangle \qquad \textrm{and} \qquad
f \mapsto \langle f, \Phi_n^{\alpha,\beta}\rangle
\end{gather*}
are bounded from $L^p(-\pi,\pi)$ to $\mathbb{C}$. Since $\mathbb{S}_{\alpha,\beta}$ is dense in $L^p(-\pi,\pi)$,
we infer that \eqref{id} holds for $f \in L^p(-\pi,\pi)$. Now, if $\mathbb{L}_{\alpha,\beta}^{-\sigma}f \equiv 0$
for some $f \in L^p(-\pi,\pi)$, then \eqref{id} implies $\langle f, \Phi_n^{\alpha,\beta}\rangle=0$ for
all $n \ge 0$ and hence Lemma~\ref{lem:wsp} gives $f \equiv 0$.
Thus $\mathbb{L}_{\alpha,\beta}^{-\sigma}$ is one-to-one on $L^p(-\pi,\pi)$.
\end{proof}

Now we can def\/ine the Jacobi potential spaces as the ranges of the potential operators on~$L^p(-\pi,\pi)$,
\begin{gather*}
\mathcal{L}_{\alpha,\beta}^{p,s}(-\pi,\pi) := \begin{cases}
															\mathbb{L}_{\alpha,\beta}^{-s/2}\big( L^p(-\pi,\pi)\big), & \alpha+\beta \neq -1,\\
															(\id+\mathbb{L}_{\alpha,\beta})^{-s/2}\big( L^p(-\pi,\pi)\big), & \alpha+\beta = -1,
														\end{cases}												
\end{gather*}
where $p \in E({\alpha,\beta})$ and $s > 0$.
Then the formula
\begin{gather*}
\|f\|_{\mathcal{L}_{\alpha,\beta}^{p,s}(-\pi,\pi)} := \|g \|_{L^p(-\pi,\pi)}, \qquad
																		\begin{cases}
																			 f=\mathbb{L}^{-s/2}_{\alpha,\beta}g, & g \in L^p(-\pi,\pi),
																			 		\ \ \alpha+\beta \neq -1,\\																																	 f = (\id+\mathbb{L}_{\alpha,\beta})^{-s/2}g, & g \in L^p(-\pi,\pi),
																			 		\ \ \alpha+\beta = -1,
																		\end{cases}
\end{gather*}
def\/ines a norm on $\mathcal{L}_{\alpha,\beta}^{p,s}(-\pi,\pi)$ and it is straightforward to check that
$\mathcal{L}_{\alpha,\beta}^{p,s}(-\pi,\pi)$ equipped with this norm is a Banach space.

In order to give a suitable def\/inition of Sobolev spaces in the symmetrized setting we need to
understand the structure of the potential spaces. The following result describes the symmetrized potential
spaces in terms of the potential spaces related to the initial, non-symmetrized situation.
The latter spaces are def\/ined similarly as $\mathcal{L}_{\alpha,\beta}^{p,s}(-\pi,\pi)$, see \cite{L2,L3} for details.

\begin{prop}\label{pot}
Let ${\alpha,\beta}>-1$, $p\in E({\alpha,\beta})$ and  $s>0$. Then
$f\in\mathcal{L}_{\alpha,\beta}^{p,s}(-\pi,\pi)$ if and only if $f_{\rm even}^+\in \mathcal{L}_{\alpha,\beta}^{p,s}(0,\pi)$
and $f_{\rm odd}^+\in \mathcal{L}_{\alpha+1,\beta+1}^{p,s}(0,\pi)$. Moreover,
\begin{gather} \label{por}
\|f\|_{\mathcal{L}_{\alpha,\beta}^{p,s}(-\pi,\pi)}\simeq 		
\|f_{{\rm even}}^+\|_{\mathcal{L}_{\alpha,\beta}^{p,s}(0,\pi)}
	+\|f_{{\rm odd}}^+\|_{\mathcal{L}_{\alpha+1,\beta+1}^{p,s}(0,\pi)}.
\end{gather}
\end{prop}

\begin{proof}
We assume that $\alpha+\beta \neq -1$, the opposite case requires only minor modif\/ications which are left to the
reader. Let $f \in \mathcal{L}_{\alpha,\beta}^{p,s}(-\pi,\pi)$. This means that there is $g \in L^p(-\pi,\pi)$
such that $f = \mathbb{L}_{\alpha,\beta}^{-s/2}g$; then
$\|f\|_{\mathcal{L}_{\alpha,\beta}^{p,s}(-\pi,\pi)}=\|g\|_{L^p(-\pi,\pi)}$.

Assume, for the time being, that $g$ belongs also to $L^2(-\pi,\pi)$. Then from \eqref{dec} we see that
\begin{gather*}
f_{\textrm{even}} = \big( \mathbb{L}_{\alpha,\beta}^{-s/2}\big)_{\textrm{e}}  g_{\textrm{even}}, \qquad
f_{\textrm{odd}} = \big( \mathbb{L}_{\alpha,\beta}^{-s/2}\big)_{\textrm{o}}  g_{\textrm{odd}}.
\end{gather*}
Thus
\begin{gather*}
f_{\textrm{even}}^+  = \sum_{n=0}^{\infty} \big(\lambda_{n}^{\alpha,\beta}\big)^{-s/2}
	\langle g_{\textrm{even}}, \Phi_{2n}^{\alpha,\beta} \rangle \big(\Phi_{2n}^{\alpha,\beta}\big)^+
 = 2 \sum_{n=0}^{\infty} \big(\lambda_{n}^{\alpha,\beta}\big)^{-s/2}
	\big\langle g_{\textrm{even}}^+, \big(\Phi_{2n}^{\alpha,\beta}\big)^+ \big\rangle_+ \big(\Phi_{2n}^{\alpha,\beta}\big)^+ \\
\hphantom{f_{\textrm{even}}^+}{} = L_{\alpha,\beta}^{-s/2} g_{\textrm{even}}^+,	
\end{gather*}
where the last identity is a consequence of~\eqref{Phi}.
Similarly, $f_{\textrm{odd}}^+ = L_{\alpha+1,\beta+1}^{-s/2} g_{\textrm{odd}}^+$.

A general $g \in L^p(-\pi,\pi)$ can be approximated in the $L^p$ norm by functions from
$L^p \cap L^2(-\pi,\pi)$. Then combining the above with the $L^p(-\pi,\pi)$-boundedness of
$\mathbb{L}_{\alpha,\beta}^{-s/2}$ and $L^p(0,\pi)$-boundedness of $L_{\alpha,\beta}^{-s/2}$ and $L_{\alpha+1,\beta+1}^{-s/2}$,
we get
\begin{gather*}
f_{\textrm{even}}^+ = L_{\alpha,\beta}^{-s/2} g_{\textrm{even}}^+, \qquad
f_{\textrm{odd}}^+ = L_{\alpha+1,\beta+1}^{-s/2} g_{\textrm{odd}}^+,
\end{gather*}
in the general case. Since
\begin{gather*}
\|g\|_{L^p(-\pi,\pi)} \simeq \|g_{\textrm{even}}^+\|_{L^p(0,\pi)} + \|g_{\textrm{odd}}^+\|_{L^p(0,\pi)},
	\qquad g \in L^p(-\pi,\pi),
\end{gather*}
we see that $f^+_{\textrm{even}} \in \mathcal{L}_{\alpha,\beta}^{p,s}(0,\pi)$,
$f^+_{\textrm{odd}} \in \mathcal{L}_{\alpha+1,\beta+1}^{p,s}(0,\pi)$ and, moreover, \eqref{por} holds.

The opposite implication is verif\/ied along similar lines. Given a function $f$ on $(-\pi,\pi)$,
assume that $f^+_{\textrm{even}} \in \mathcal{L}_{\alpha,\beta}^{p,s}(0,\pi)$ and
$f^+_{\textrm{odd}} \in \mathcal{L}_{\alpha+1,\beta+1}^{p,s}(0,\pi)$. Then
$f^+_{\textrm{even}} = L_{\alpha,\beta}^{-s/2}h$ and $f^+_{\textrm{odd}} = L_{\alpha+1,\beta+1}^{-s/2}\widetilde{h}$ for
some $h,\widetilde{h} \in L^p(0,\pi)$. Extending $h$ and $\widetilde{h}$ to even and odd functions on
$(-\pi,\pi)$, respectively, we let $g$ be the sum of these extensions. Then we f\/ind that
$f^+_{\textrm{even}} = (\mathbb{L}_{\alpha,\beta}^{-s/2}g)_{\textrm{even}}^+$,
$f^+_{\textrm{odd}} = (\mathbb{L}_{\alpha,\beta}^{-s/2}g)_{\textrm{odd}}^+$, and consequently
$f = \mathbb{L}_{\alpha,\beta}^{-s/2}g$ with $g \in L^p(-\pi,\pi)$. Thus $f\in \mathcal{L}_{\alpha,\beta}^{p,s}(-\pi,\pi)$.
\end{proof}

\section{Sobolev spaces} \label{sec:sob}

Our aim in this section is to establish a suitable def\/inition of Sobolev spaces in the symmetrized setting.
Here ``suitable'' means existence of an isomorphism between the Sobolev spaces and the potential spaces
with properly chosen parameters. Note that such an isomorphism gives also a characterization of the potential
spaces with some parameters in terms of appropriate higher-order `derivatives'.

According to a general concept, Sobolev spaces $\mathbb{W}_{\alpha,\beta}^{p,m}$, $m \ge 1$ integer, associated with
$\mathbb{L}_{\alpha,\beta}$, should be def\/ined by
\begin{gather*}
\mathbb{W}_{\alpha,\beta}^{p,m} := \big\{ f \in L^p(-\pi,\pi)\colon \mathfrak{D}_{\alpha,\beta}^{(k)}f \in L^p(-\pi,\pi),\,
	k=1,\ldots,m \big\}
\end{gather*}
and equipped with the norm
\begin{gather*}
\|f\|_{\mathbb{W}_{\alpha,\beta}^{p,m}} := \sum_{k=0}^{m} \big\|\mathfrak{D}_{\alpha,\beta}^{(k)}f\big\|_{L^p(-\pi,\pi)}.
\end{gather*}
Here $\mathfrak{D}_{\alpha,\beta}^{(k)}$ is a suitably def\/ined dif\/ferential-dif\/ference operator of order~$k$
playing the role of higher-order derivative, with the dif\/ferentiation understood in a weak sense; we use the convention $\mathfrak{D}_{\alpha,\beta}^{(0)}:=\id$.
So the question is how to choose $\mathfrak{D}_{\alpha,\beta}^{(k)}$.

It turns out that the seemingly most natural choice $\mathfrak{D}_{\alpha,\beta}^{(k)} = \mathbb{D}_{\alpha,\beta}^k$ is
not appropriate. Another quite natural attempt is to mimic the variable index derivatives, which lead
to a good def\/inition of Sobolev spaces in the non-symmetrized setting, see \cite[Section~2]{L2}.
Unfortunately, taking $\mathfrak{D}_{\alpha,\beta}^{(k)} = \mathbb{D}_{\alpha+k-1,\beta+k-1}\circ \cdots\circ
\mathbb{D}_{\alpha+1,\beta+1} \circ \mathbb{D}_{\alpha,\beta}$ is inappropriate as well. Counterexamples for these choices are
discussed at the end of this section.

To f\/ind suitable higher-order `derivatives' in the symmetrized setting we f\/irst introduce the variable index
higher-order `derivatives'
\begin{gather*}
\mathfrak{d}_{\alpha,\beta}^{(k)} := D_{\alpha+k-1,\beta+k-1} \circ \cdots \circ D_{\alpha+1,\beta+1} \circ D_{\alpha,\beta},
\qquad k \ge 1,
\end{gather*}
acting on functions on $(-\pi,\pi)$ or $(0,\pi)$; we set $\mathfrak{d}_{\alpha,\beta}^{(0)}:=\id$.
In \cite[Theorem A]{L2} we proved that in the non-symmetrized Jacobi function setting the Sobolev spaces
\begin{gather*}
W_{\alpha,\beta}^{p,m}(0,\pi) =
	\big\{ h \in L^p(0,\pi) \colon  \mathfrak{d}_{\alpha,\beta}^{(k)}h \in L^p(0,\pi), \, k=1,\ldots,m \big\}
\end{gather*}
equipped with the norm
\begin{gather*}
\|h\|_{W_{\alpha,\beta}^{p,m}(0,\pi)} = \sum_{k=0}^m \big\|\mathfrak{d}_{\alpha,\beta}^{(k)}h\big\|_{L^p(0,\pi)},
\end{gather*}
are isomorphic to the potential spaces $\mathcal{L}_{\alpha,\beta}^{p,m}(0,\pi)$.
Combining this result with Proposition~\ref{pot} we get the following.
\begin{prop}\label{sob}
Let ${\alpha,\beta}>-1$, $p\in E({\alpha,\beta})$ and $m\ge1$ be integer. Then
$f\in\mathcal{L}_{\alpha,\beta}^{p,m}(-\pi,\pi)$ if and only if
$f_{\rm even}^+\in W_{\alpha,\beta}^{p,m}(0,\pi)$ and $f_{\rm odd}^+\in W_{\alpha+1,\beta+1}^{p,m}(0,\pi)$.
Moreover,
\begin{gather*}
\|f\|_{\mathcal{L}_{\alpha,\beta}^{p,m}(-\pi,\pi)}\simeq
 \|f_{{\rm even}}^+\|_{W_{\alpha,\beta}^{p,m}(0,\pi)}+\|f_{{\rm odd}}^+\|_{W_{\alpha+1,\beta+1}^{p,m}(0,\pi)}.
\end{gather*}
\end{prop}

This motivates the following def\/inition of the higher-order `derivatives' $\mathfrak{D}_{\alpha,\beta}^{(k)}$.
\begin{defin} \label{def}
For k=0,1,2,\ldots, let
\begin{gather*}
\mathfrak{D}_{\alpha,\beta}^{(k)} f := \mathfrak{d}_{\alpha,\beta}^{(k)} f_{{\rm even}}
			+\mathfrak{d}_{\alpha+1,\beta+1}^{(k)} f_{{\rm odd}}.
\end{gather*}
\end{defin}

Note that the `derivatives' $\mathfrak{D}_{\alpha,\beta}^{(k)}$ are counterintuitive from the point of view of
the symmetrization concept, since they do not express via compositions of the symmetrized `derivative'~$\mathbb{D}_{\alpha,\beta}$. Nevertheless, we have the following result.

\begin{thm} \label{thm:sob1}
Let ${\alpha,\beta}>-1$, $p \in E({\alpha,\beta})$ and $m \ge 1$ be integer. Then
\begin{gather*}
\mathbb{W}_{\alpha,\beta}^{p,m} = \mathcal{L}_{\alpha,\beta}^{p,m}(-\pi,\pi)
\end{gather*}
in the sense of isomorphism of Banach spaces.
\end{thm}

\begin{proof}
Let $k \ge 0$.
For symmetry reasons, we have
$\mathfrak{D}_{\alpha,\beta}^{(k)} f\in L^p(-\pi,\pi)$ if and only if
$\mathfrak{d}_{\alpha,\beta}^{(k)} f_{\textrm{even}}^+\in L^p(0,\pi)$
and $\mathfrak{d}_{\alpha+1,\beta+1}^{(k)} f_{\textrm{odd}}^+\in L^p(0,\pi)$. Furthermore,
\begin{gather*}
\big\|\mathfrak{D}_{\alpha,\beta}^{(k)}f\big\|_{L^p(-\pi,\pi)}\simeq
	\big\|\mathfrak{d}_{\alpha,\beta}^{(k)}f_{\textrm{even}}^+\big\|_{L^p(0,\pi)}
	+\big\|\mathfrak{d}_{\alpha+1,\beta+1}^{(k)}f_{\textrm{odd}}^+\big\|_{L^p(0,\pi)}.
\end{gather*}
Thus the assertion follows from Proposition \ref{sob}.
\end{proof}

In the remaining part of this section we look closer at the two already mentioned, seemingly more natural
concepts of Sobolev spaces in the symmetrized setting, which in general fail to be isomorphic with the
corresponding potential spaces. For $m \ge 1$ denote
\begin{gather*}
\mathcal{W}_{\alpha,\beta}^{p,m}  := \big\{ f \in L^p(-\pi,\pi) \colon \mathbb{D}_{\alpha,\beta}^k f \in L^p(-\pi,\pi), \,
	k=1,\ldots,m\big\}, \\
\mathfrak{W}_{\alpha,\beta}^{p,m}  := \big\{ f \in L^p(-\pi,\pi) \colon
	\mathbb{D}_{\alpha+k-1,\beta+k-1} \cdots \mathbb{D}_{\alpha+1,\beta+1}\mathbb{D}_{\alpha,\beta} f
	\in L^p(-\pi,\pi),\, k=1,\ldots,m\big\},
\end{gather*}
and equip these spaces with the natural norms.

\begin{prop} \label{prop:neg}
Let ${\alpha,\beta} > -1$ and $p \in E({\alpha,\beta})$. For every ${\alpha,\beta} < 1/p-1/2$ there exists
$f \in \mathcal{W}_{\alpha,\beta}^{p,1} = \mathfrak{W}_{\alpha,\beta}^{p,1}$ such that
$f \notin \mathcal{L}_{\alpha,\beta}^{p,1}(-\pi,\pi)$.
Furthermore, if $\alpha \le -1/p+1/2$ or $\beta \le -1/p+1/2$ then there exists
$g\in \mathcal{L}_{\alpha,\beta}^{p,2}(-\pi,\pi)$ such that $g\notin \mathfrak{W}_{\alpha,\beta}^{p,2}$.
\end{prop}

\begin{proof}
Let
\begin{gather*}
f(\theta) = \sign(\theta)\Psi^{-\alpha-1,-\beta-1}(\theta) = \sign(\theta)
		 \left|\sin\frac{\theta}2\right|^{-\alpha-1/2}\left(\cos\frac{\theta}2\right)^{-\beta-1/2}.
\end{gather*}
Since $f$ is odd, we have $(\mathbb{D}_{\alpha,\beta}f)^+=D_{\alpha,\beta}^*f^+$ and \cite[formula~(9)]{L2}
reveals that $D_{\alpha,\beta}^*f^+$ vanishes. Therefore $f^+,D_{\alpha,\beta}^{*}f^+ \in L^p(0,\pi)$ for ${\alpha,\beta}$ in question.
By symmetry it follows that $f \in \mathcal{W}_{\alpha,\beta}^{p,1}$.

On the other hand, with the aid of \cite[formula~(8)]{L2} we f\/ind that
\begin{gather*}
\big(\mathfrak{D}_{\alpha,\beta}^{(1)}f\big)^+(\theta) = \big(\mathfrak{d}_{\alpha+1,\beta+1}^{(1)}f\big)^+(\theta)
\gtrsim \theta^{-\alpha-3/2}(\pi-\theta)^{-\beta-3/2}, \qquad \theta \in (0,\pi).
\end{gather*}
Thus $(\mathfrak{D}_{\alpha,\beta}^{(1)}f)^+ \notin L^p(0,\pi)$, in view of the assumption $p \in E({\alpha,\beta})$.
This implies $f \notin \mathbb{W}_{\alpha,\beta}^{p,1}$. By Theorem \ref{thm:sob1},
$f \notin \mathcal{L}_{\alpha,\beta}^{p,1}(-\pi,\pi)$.

To prove the second claim let
\begin{gather*}
g(\theta) = \sign(\theta)\Psi^{\alpha+1,\beta+1}(\theta) = \sign(\theta)
		 \left|\sin\frac{\theta}2\right|^{\alpha+3/2}\left(\cos\frac{\theta}2\right)^{\beta+3/2}.
\end{gather*}
Since $g$ is odd, we have
$\big(\mathbb{D}_{\alpha+1,\beta+1}\mathbb{D}_{\alpha,\beta}g\big)^+=D_{\alpha+1,\beta+1}D_{\alpha,\beta}^*g^+$.
Using \cite[formulas~(8) and~(9)]{L2} we get
\begin{gather*}
(\mathbb{D}_{\alpha+1,\beta+1}\mathbb{D}_{\alpha,\beta}g)^+(\theta)
	\gtrsim \theta^{\alpha-1/2}(\pi-\theta)^{\beta-1/2}, \qquad \theta \in (0,\pi).
\end{gather*}
Consequently, for the assumed range of ${\alpha,\beta}$ we have that $(\mathbb{D}_{\alpha+1,\beta+1}\mathbb{D}_{\alpha,\beta}g)^+$ is
not in $L^p(0,\pi)$. Therefore, by symmetry, $g\notin \mathfrak{W}_{\alpha,\beta}^{p,2}$.

On the other hand $\big(\mathfrak{D}_{\alpha,\beta}^{(1)}g\big)^+ =
 \big(\mathfrak{d}_{\alpha+1,\beta+1}^{(1)}g\big)^+=D_{\alpha+1,\beta+1}g^+ =0$,
by \cite[formula~(8)]{L2}. Consequently,
$\big(\mathfrak{D}_{\alpha,\beta}^{(2)}g\big)^+ =D_{\alpha+2,\beta+2}D_{\alpha+1,\beta+1}g^+ =0$.
Since $g^+ \in L^p(0,\pi)$, using again the symmetry we see that $g\in \mathbb{W}_{\alpha,\beta}^{p,2}$.
This together with Theorem~\ref{thm:sob1} implies $g\in \mathcal{L}_{\alpha,\beta}^{p,2}(-\pi,\pi)$.
\end{proof}

Although Proposition~\ref{prop:neg} shows that the spaces $\mathcal{W}_{\alpha,\beta}^{p,m}$ and
$\mathcal{L}_{\alpha,\beta}^{p,m}(-\pi,\pi)$ do not coincide in general, one might still wonder what the relation
between them is, if any. The answer is given by the next result.

\begin{prop} \label{prop:incl}
Let ${\alpha,\beta} > -1$, $p \in E({\alpha,\beta})$ and $m \ge 1$ be integer. Then
\begin{gather*}
\mathcal{L}_{\alpha,\beta}^{p,m}(-\pi,\pi) \subset \mathcal{W}_{\alpha,\beta}^{p,m}
\end{gather*}
in the sense of embedding of Banach spaces.
\end{prop}

The proof of Proposition \ref{prop:incl} involves higher-order Riesz transforms of the following form.
For $k \ge 1$ integer, let
\begin{gather*}
\mathbb{R}_{\alpha,\beta}^{k}=
	\begin{cases}
		\mathbb{D}_{\alpha,\beta}^{k}\mathbb{L}_{\alpha,\beta}^{-k/2}, &  \alpha+\beta\neq-1,\\
		\mathbb{D}_{\alpha,\beta}^{k}(\id+\mathbb{L}_{\alpha,\beta})^{-k/2},& \alpha+\beta=-1.
	\end{cases}
\end{gather*}
Clearly, $\mathbb{R}_{\alpha,\beta}^k$ is well def\/ined on~$\mathbb{S}_{\alpha,\beta}$. But we also need to know that each~$\mathbb{R}_{\alpha,\beta}^k$, $k \ge 1$, extends to a bounded operator on $L^p(-\pi,\pi)$.

\begin{lem}\label{riesz}
Let ${\alpha,\beta} > -1$, $p \in E({\alpha,\beta})$ and $k\ge 1$. Then the operator
\begin{gather*}
f\mapsto \mathbb{R}_{\alpha,\beta}^k f,\qquad f\in \mathbb{S}_{\alpha,\beta},
\end{gather*}
extends uniquely to a bounded linear operator on $L^p(-\pi,\pi)$.
\end{lem}

Assuming that this result holds, we now give a short proof of Proposition~\ref{prop:incl}.
Lemma~\ref{riesz} will be shown subsequently.
\begin{proof}[Proof of Proposition \ref{prop:incl}]
We may assume that $\alpha+\beta \neq-1$, since treatment of the opposite case is analogous.
Let $f\in\mathcal{L}_{\alpha,\beta}^{p,m}(-\pi,\pi)$. Then $f=\mathbb{L}_{\alpha,\beta}^{-m/2}g$ for some $g\in L^p(-\pi,\pi).$
By the $L^p$-boundedness of $\mathbb{L}_{\alpha,\beta}^{-(m-k)/2}$ (see Proposition \ref{ogr}) and Lemma \ref{riesz},
for any $0\le k\le m$ we have
\begin{gather*}
\big\|\mathbb{D}^k_{\alpha,\beta}f\big\|_{L^p(-\pi,\pi)}   =
\big\|\mathbb{D}_{\alpha,\beta}^{k}\mathbb{L}_{\alpha,\beta}^{-m/2}g\big\|_{L^p(-\pi,\pi)}=
\big\|\mathbb{R}_{\alpha,\beta}^{k}\mathbb{L}_{\alpha,\beta}^{-(m-k)/2}g\big\|_{L^p(-\pi,\pi)} \\
 \hphantom{\big\|\mathbb{D}^k_{\alpha,\beta}f\big\|_{L^p(-\pi,\pi)}}{}
  \lesssim \|g\|_{L^p(-\pi,\pi)}= \|f\|_{\mathcal{L}_{\alpha,\beta}^{p,m}(-\pi,\pi)},
\end{gather*}
where $\mathbb{R}_{\alpha,\beta}^k$ stands for the extension provided by Lemma~\ref{riesz} (with the natural
interpretation $\mathbb{R}_{\alpha,\beta}^0=\id$). The second identity above is easily justif\/ied when
$g \in \mathbb{S}_{\alpha,\beta}$, and then it carries over to general~$g$ by continuity.
The conclusion follows.
\end{proof}

It remains to prove Lemma \ref{riesz}. The argument relies on a multiplier-transplantation theorem due
to Muckenhoupt~\cite{M}, see \cite[Lemma~2.1]{L2}. Here we merely sketch the proof, leaving the details to
interested readers.
\begin{proof}[{Proof of Lemma~\ref{riesz}}]
Assume that $\alpha+\beta \neq -1$ (the opposite case is similar) and take $f \in \mathbb{S}_{\alpha,\beta}$.
We have
\begin{gather*}
\mathbb{R}_{\alpha,\beta}^k f = -(-1)^{\langle k \rangle}
	\underbrace{\cdots D_{\alpha,\beta}^*D_{\alpha,\beta}}_{k\;\textrm{components}}
	 \big(\mathbb{L}_{\alpha,\beta}^{-k/2}f\big)_{\textrm{even}}
+ (-1)^{\langle k \rangle} \underbrace{\cdots D_{\alpha,\beta}D_{\alpha,\beta}^*}_{k\; \textrm{components}}	
	\big(\mathbb{L}_{\alpha,\beta}^{-k/2}f\big)_{\textrm{odd}}.
\end{gather*}
This is the decomposition of $\mathbb{R}_{\alpha,\beta}^k f$ into its even and odd parts, respectively, or vice versa,
depending on whether $k$ is even or odd. Since (see the proof of Proposition~\ref{pot})
\begin{gather*}
\big(\mathbb{L}_{\alpha,\beta}^{-k/2}f\big)_{\textrm{even}}^+ = L_{\alpha,\beta}^{-k/2} f_{\textrm{even}}^+, \qquad
\big(\mathbb{L}_{\alpha,\beta}^{-k/2}f\big)_{\textrm{odd}}^+ = L_{\alpha+1,\beta+1}^{-k/2} f_{\textrm{odd}}^+,
\end{gather*}
it suf\/f\/ices to show the bounds
\begin{gather}
\big\|\underbrace{\cdots D_{\alpha,\beta}^*D_{\alpha,\beta}}_{k\;\textrm{components}} L_{\alpha,\beta}^{-k/2}h\big\|_{L^p(0,\pi)}
	  \lesssim \|h\|_{L^p(0,\pi)}, \qquad h \in \spann\big\{\phi_n^{\alpha,\beta}\colon n\ge 0\big\}, \label{b1}\\
\big\|\underbrace{\cdots D_{\alpha,\beta}D_{\alpha,\beta}^*}_{k\; \textrm{components}} L_{\alpha+1,\beta+1}^{-k/2}h\big\|_{L^p(0,\pi)}
	  \lesssim \|h\|_{L^p(0,\pi)}, \qquad h \in \spann\big\{\phi_n^{\alpha+1,\beta+1}\colon n \ge 0\big\}. \label{b2}
\end{gather}
Here \eqref{b1} is contained in \cite[Proposition~4.2]{L2}, since the underlying operator
coincides with the Riesz transform $\mathcal{R}_{\alpha,\beta}^k$ considered in~\cite{L2}. So it remains to
verify~\eqref{b2}.

Taking into account \cite[formulas~(5) and~(6)]{L2} one f\/inds that
\begin{gather*}
\underbrace{\cdots D_{\alpha,\beta}D_{\alpha,\beta}^*}_{k\; \textrm{components}} L_{\alpha+1,\beta+1}^{-k/2}h =
	\sum_{n=0}^{\infty} \left( 1 - \frac{A_{\alpha,\beta}^2}{\lambda_{n+1}^{\alpha,\beta}}\right)^{-k/2}
	\big\langle h,\phi_n^{\alpha+1,\beta+1}\big\rangle_{+}
	\begin{cases}
		\phi_n^{\alpha+1,\beta+1}, & k \; \textrm{even},\\
		-\phi_{n+1}^{\alpha,\beta}, & k \; \textrm{odd}.
	\end{cases}
\end{gather*}
Now \eqref{b2} follows from a special case of Muckenhoupt's multiplier-transplantation theorem
\cite[Lem\-ma~2.1]{L2}; see, e.g., the proof of \cite[Proposition~3.4]{L2}.
\end{proof}

\section[Characterization of potential spaces via fractional square functions]{Characterization of potential spaces\\ via fractional square functions}

In this section we give necessary and suf\/f\/icient conditions, expressed in terms of suitably def\/ined fractional
square functions, for a function to belong to the potential space $\mathcal{L}_{\alpha,\beta}^{p,s}(-\pi,\pi)$.
For the sake of brevity, we restrict our main attention to the case $\alpha+\beta \neq -1$.
Nevertheless, after a~slight modif\/ication the result is valid also when $\alpha + \beta = -1$.
This issue is discussed at the end of this section.

Let $\{\mathbb{H}_t^{\alpha,\beta}\colon t \ge 0\}$ be the symmetrized Poisson--Jacobi semigroup, i.e., the semigroup
of operators generated by $-\mathbb{L}_{\alpha,\beta}^{1/2}$. In view of the spectral theorem, for
$f \in L^2(-\pi,\pi)$ and $t \ge 0$ we have
\begin{gather*}
\mathbb{H}_t^{\alpha,\beta}f
	= \sum_{n=0}^{\infty} \exp\Big({-}t\sqrt{\lambda_{\langle n \rangle }^{\alpha,\beta}} \Big)
	\langle f, \Phi_n^{\alpha,\beta} \rangle \Phi_n^{\alpha,\beta},
\end{gather*}
the convergence being in $L^2(-\pi,\pi)$.
By means of~\eqref{Phi} and \cite[Estimate~(1)]{L2} one verif\/ies that for $t >0$ the above series
converges in fact pointwise in $(-\pi,\pi){\setminus} \{0\}$ and, moreover, may serve as a pointwise def\/inition of $\mathbb{H}_{t}^{\alpha,\beta}f$ on $(-\pi,\pi){\setminus} \{0\}$, $t > 0$, for $f \in L^p(-\pi,\pi)$, $p > p'({\alpha,\beta})$. In the latter case, the resulting function
$\mathbb{H}_t^{\alpha,\beta}f(\theta)$ is smooth in $(t,\theta) \in (0,\infty) \times [(-\pi,\pi){\setminus} \{0\}]$. There is also
an integral representation of $\mathbb{H}_t^{\alpha,\beta}f$, $t >0$, for $f$ as above, but it will not be needed for
our purposes.

Following Segovia and Wheeden \cite{Seg&Whe} and Betancor et al.~\cite{betsq}, see also~\cite{L3},
we consider the fractional square function
\begin{gather*}
\mathfrak{g}_{\alpha,\beta}^{\gamma,k}f(\theta) = \left( \int_{0}^{\infty} \left| t^{k-\gamma}
	\frac{\partial^k}{\partial t^k} \mathbb{H}_t^{\alpha,\beta}f(\theta)\right|^2 \frac{dt}t \right)^{1/2},
		\qquad \theta \in (-\pi,\pi){\setminus} \{0\},
\end{gather*}
where $0 < \gamma < k$ and $k=1,2,\ldots$.
Notice that $\mathfrak{g}_{\alpha,\beta}^{\gamma,k}f$ is well def\/ined pointwise whenever $f \in L^p(-\pi,\pi)$
and $p > p'({\alpha,\beta})$. An analogue of $\mathfrak{g}_{\alpha,\beta}^{\gamma,k}$ was investigated in the non-symmetrized
Jacobi function setting in~\cite{L3} and denoted by $g_{\alpha,\beta}^{\gamma,k}$ there. Recall that
\begin{gather*}
{g}_{\alpha,\beta}^{\gamma,k}h(\theta) = \left( \int_{0}^{\infty} \left| t^{k-\gamma}
	\frac{\partial^k}{\partial t^k} {H}_t^{\alpha,\beta}h(\theta)\right|^2 \frac{dt}t \right)^{1/2},
		\qquad \theta \in (0,\pi),
\end{gather*}
where $\gamma$ and $k$ are as before, and $\{H_t^{\alpha,\beta}\}$ is the Poisson--Jacobi semigroup. See~\cite{L3}
for more details on $H_t^{\alpha,\beta}$ and $g_{\alpha,\beta}^{\gamma,k}$.

A simple combination of Proposition~\ref{pot} and \cite[Theorem~4.1]{L3} allows us to get the following
description of the symmetrized potential spaces $\mathcal{L}_{\alpha,\beta}^{p,s}(-\pi,\pi)$ in terms of the
non-symmetrized square functions $g_{\alpha,\beta}^{\gamma,k}$.

\begin{prop}\label{sq}
Let ${\alpha,\beta}>-1$ be such that $\alpha+\beta \neq -1$ and let $p\in E({\alpha,\beta})$.
Fix $0 < \gamma < k$ with $k \in \mathbb{N}$. Then $f\in\mathcal{L}_{\alpha,\beta}^{p,\gamma}(-\pi,\pi)$
if and only if $f \in L^p(-\pi,\pi)$ and $g_{\alpha,\beta}^{\gamma,k}f_{{\rm even}}^+, g_{\alpha+1,\beta+1}^{\gamma,k}f_{{\rm odd}}^+ \in L^p(0,\pi)$.
Moreover,
\begin{gather*}
\|f\|_{\mathcal{L}_{\alpha,\beta}^{p,\gamma}(-\pi,\pi)} \simeq
	\big\|g_{\alpha,\beta}^{\gamma,k}f_{{\rm even}}^+\big\|_{L^p(0,\pi)}
	+\big\|g_{\alpha+1,\beta+1}^{\gamma,k}f_{{\rm odd}}^+\big\|_{L^p(0,\pi)},
		\qquad f \in \mathcal{L}_{\alpha,\beta}^{p,\gamma}(-\pi,\pi).
\end{gather*}
\end{prop}

This leads to the following characterization of the symmetrized potential spaces.

\begin{thm} \label{thm:char}
Let ${\alpha,\beta}>-1$ be such that $\alpha+\beta \neq -1$ and let $p \in E({\alpha,\beta})$. Fix $0<\gamma < k$ with $k \in \mathbb{N}$.
Then $f\in\mathcal{L}_{\alpha,\beta}^{p,\gamma}(-\pi,\pi)$ if and only if $f\in L^p(-\pi,\pi)$ and
$\mathfrak{g}_{\alpha,\beta}^{\gamma,k}f\in L^p(-\pi,\pi)$. Moreover,
\begin{gather*}
  \|f\|_{\mathcal{L}_{\alpha,\beta}^{p,\gamma}(-\pi,\pi)}\simeq
  	\big\|\mathfrak g_{\alpha,\beta}^{\gamma,k}f\big\|_{L^p(-\pi,\pi)},
  		\qquad f \in \mathcal{L}_{\alpha,\beta}^{p,\gamma}(-\pi,\pi).
\end{gather*}
\end{thm}

\begin{proof}
Taking into account Proposition \ref{sq}, it suf\/f\/ices to show the following. Given $p \in E({\alpha,\beta})$,
\begin{gather*}
\big\| \mathfrak{g}_{\alpha,\beta}^{\gamma,k}f\big\|_{L^p(-\pi,\pi)} \simeq
	\big\|g_{\alpha,\beta}^{\gamma,k}f_{\textrm{even}}^+\big\|_{L^p(0,\pi)} +
		\big\|g_{\alpha+1,\beta+1}^{\gamma,k} f_{\textrm{odd}}^+\big\|_{L^p(0,\pi)},
\end{gather*}
uniformly in $f\in L^p(-\pi,\pi)$, possibly with inf\/inite values on both sides for some $f$.

To proceed, observe that
\begin{gather}
\big( \mathbb{H}_t^{\alpha,\beta} f_{\textrm{even}}\big)^+(\theta)   =
	\big( \mathbb{H}_t^{\alpha,\beta}f \big)_{\textrm{even}}^+(\theta) = H_t^{\alpha,\beta}f_{\textrm{even}}^+(\theta),
		\qquad \theta \in (0,\pi), \label{crse} \\
\big( \mathbb{H}_t^{\alpha,\beta} f_{\textrm{odd}}\big)^+(\theta)   =
	\big( \mathbb{H}_t^{\alpha,\beta}f \big)_{\textrm{odd}}^+(\theta) = H_t^{\alpha+1,\beta+1}f_{\textrm{odd}}^+(\theta),
		\qquad \theta \in (0,\pi). \label{crso}
\end{gather}
These identities are easily verif\/ied by means of the series representations of $\mathbb{H}_t^{\alpha,\beta}$
and $H_t^{\alpha,\beta}$, since the relevant series converge pointwise.

Next, we claim that
\begin{gather} \label{bg}
\mathfrak{g}_{\alpha,\beta}^{\gamma,k}f(\theta) \le \mathfrak{g}_{\alpha,\beta}^{\gamma,k}f_{\textrm{even}}(\theta)
	+ \mathfrak{g}_{\alpha,\beta}^{\gamma,k} f_{\textrm{odd}}(\theta) \le
	\mathfrak{g}_{\alpha,\beta}^{\gamma,k}f(\theta) + \mathfrak{g}_{\alpha,\beta}^{\gamma,k}f(-\theta).
\end{gather}
Here the lower bound is clear, since $\mathfrak{g}_{\alpha,\beta}^{\gamma,k}$ is sublinear.
To see the upper bound, we f\/irst observe that the operators $\mathbb{H}_t^{\alpha,\beta}$, $t>0$, commute with ref\/lections.
It is enough to verify this on $\mathbb{S}_{\alpha,\beta}$. Since $\check{\Phi}_n^{\alpha,\beta} = \Phi_n^{\alpha,\beta}$ for
$n$ even and $\check{\Phi}_n^{\alpha,\beta} = -\Phi_n^{\alpha,\beta}$ for $n$ odd, we can write
\begin{gather*}
\mathbb{H}_t^{\alpha,\beta} \check{\Phi}_{n}^{\alpha,\beta}(\theta)   =
	\exp\Big({-}t\sqrt{\lambda_{\langle n \rangle }^{\alpha,\beta}} \Big) (-1)^n \Phi_{n}^{\alpha,\beta}(\theta)
	= \exp\Big({-}t\sqrt{\lambda_{\langle n \rangle }^{\alpha,\beta}} \Big) \Phi_n^{\alpha,\beta}(-\theta)   = \mathbb{H}_t^{\alpha,\beta} \Phi_n^{\alpha,\beta}(-\theta).
\end{gather*}
Consequently, $\mathfrak{g}_{\alpha,\beta}^{\gamma,k}\check{f}(\theta) = \mathfrak{g}_{\alpha,\beta}^{\gamma,k}f(-\theta)$.
Using now the identities $f_{\textrm{even}} = (f+\check{f})/2$ and $f_{\textrm{odd}} = (f-\check{f})/2$
we arrive at the upper estimate in \eqref{bg}. The claim follows.

Using \eqref{bg} and then \eqref{crse}, \eqref{crso}, together with the symmetries of
$\mathbb{H}_t^{\alpha,\beta}f_{\textrm{even}}$ and $\mathbb{H}_t^{\alpha,\beta}f_{\textrm{odd}}$, we obtain
\begin{gather*}
\big\| \mathfrak{g}_{\alpha,\beta}^{\gamma,k}f\big\|_{L^p(-\pi,\pi)}   \simeq
	\big\|\mathfrak{g}_{\alpha,\beta}^{\gamma,k}f_{\textrm{even}}\big\|_{L^p(-\pi,\pi)}
	+ \big\|\mathfrak{g}_{\alpha,\beta}^{\gamma,k}f_{\textrm{odd}}\big\|_{L^p(-\pi,\pi)} \\
\hphantom{\big\| \mathfrak{g}_{\alpha,\beta}^{\gamma,k}f\big\|_{L^p(-\pi,\pi)} }{}
  \simeq \big\|g_{\alpha,\beta}^{\gamma,k}f_{\textrm{even}}^+\big\|_{L^p(0,\pi)}
	+ \big\|g_{\alpha+1,\beta+1}^{\gamma,k}f_{\textrm{odd}}^+\big\|_{L^p(0,\pi)}.
\end{gather*}
This f\/inishes the proof.
\end{proof}

In the remaining part of this section we deal with the case $\alpha + \beta = -1$, which is not covered by
Theorem \ref{thm:char}. Actually, only a slight modif\/ication is needed, and this is connected with the
fact that for $\alpha+\beta=-1$ the potential spaces are def\/ined via the Bessel type potentials
$(\id + \mathbb{L}_{\alpha,\beta})^{-s/2}$. The main idea of what follows is taken from \cite[Section~4]{L3}.
Here we give only a general outline and state the relevant result. The details consist of a combination
of the facts and results described at the end of \cite[Section~4]{L3} and the arguments already used
in this section. This is left to interested readers.

Consider the modif\/ied symmetrized Jacobi `Laplacian'
\begin{gather*}
\widetilde{\mathbb{L}}_{\alpha,\beta}:=\big(\id + \mathbb{L}_{\alpha,\beta}^{1/2}\big)^2
\end{gather*}
and the related modif\/ied Riesz type potentials $\widetilde{\mathbb{L}}_{\alpha,\beta}^{-s/2}$.
Since the spectrum of $\widetilde{\mathbb{L}}_{\alpha,\beta}$ is separated from $0$,
the latter operators are well def\/ined spectrally. Moreover, they
extend to bounded operators on $L^p(-\pi,\pi)$,
$p \in E({\alpha,\beta})$. Since these extensions are one-to-one on $L^p(-\pi,\pi)$, one can def\/ine for $p \in E({\alpha,\beta})$
the modif\/ied potential spaces as
\begin{gather*}
\widetilde{\mathcal{L}}_{\alpha,\beta}^{p,s}(-\pi,\pi) :=
 \widetilde{\mathbb{L}}_{\alpha,\beta}^{-s/2}\big(L^p(-\pi,\pi)\big),
\end{gather*}
with the norm $\|f\|_{\widetilde{\mathcal{L}}_{\alpha,\beta}^{p,s}(-\pi,\pi)} := \|g\|_{L^p(-\pi,\pi)}$, where $g$ is such that
$f = \widetilde{\mathbb{L}}_{\alpha,\beta}^{-s/2}g$. These are Banach spaces, and the crucial fact is that they are
isomorphic to the non-modif\/ied potential spaces $\mathcal{L}_{\alpha,\beta}^{p,s}(-\pi,\pi)$.

The Poisson semigroup related to $\widetilde{\mathbb{L}}_{\alpha,\beta}$ is generated by
$-\id-\mathbb{L}_{\alpha,\beta}^{1/2}$, hence it has the form $\{e^{-t}\mathbb{H}_t^{\alpha,\beta}\}$.
Consequently, the relevant square function is given by
\begin{gather*}
\widetilde{\mathfrak g}_{\alpha,\beta}^{\,\gamma,k}f(\theta) =\left(\int_0^{\infty}\left|t^{k-\gamma}
	\frac{\partial^k}{\partial t^k} \big[e^{-t}\mathbb{H}_t^{\alpha,\beta}f(\theta)\big]\right|^2\frac{dt}{t}\right)^{1/2},
		\qquad \theta \in (-\pi,\pi) {\setminus} \{0\},
\end{gather*}
where $0 < \gamma < k$ and $k=1,2,\ldots$.

The desired alternative characterization of $\mathcal{L}_{\alpha,\beta}^{p,s}(-\pi,\pi)$ reads as follows.
\begin{thm} \label{equivfun'}
Let ${\alpha,\beta} > -1$ and $p\in E({\alpha,\beta})$. Fix $0 < \gamma < k$ with $k \in \mathbb{N}$.
Then $f\in\mathcal{L}_{\alpha,\beta}^{p,\gamma}(-\pi,\pi)$ if and only if $f\in L^p(-\pi,\pi)$ and
$\widetilde{\mathfrak{g}}_{\alpha,\beta}^{\,\gamma,k}f\in L^p(-\pi,\pi)$. Moreover,
\begin{gather*}
\|f\|_{\mathcal{L}_{\alpha,\beta}^{p,\gamma}(-\pi,\pi)} \simeq
	\big\|\widetilde{\mathfrak{g}}_{\alpha,\beta}^{\,\gamma,k}f\big\|_{L^p(-\pi,\pi)},
		\qquad f\in \mathcal{L}_{\alpha,\beta}^{p,\gamma}(-\pi,\pi).
\end{gather*}
\end{thm}

\section{Structural and embedding theorems for potential spaces} \label{sec:struc}

We now show some results revealing relations between the symmetrized potential spaces with dif\/ferent
parameters and also establishing mapping properties of certain operators with respect to the potential
spaces. At the end of this section we state an analogue of the classical Sobolev embedding theorem
for the symmetrized potential spaces.

The results of Section~\ref{sec:sob} suggest the following alternative def\/inition of Riesz transforms in the
symmetrized setting. For $k \ge 1$ integer, we let
\begin{gather*}
\mathfrak{R}_{\alpha,\beta}^{k}=
	\begin{cases}
		\mathfrak{D}_{\alpha,\beta}^{(k)}\mathbb{L}_{\alpha,\beta}^{-k/2}, &  \alpha+\beta\neq-1,\\
		\mathfrak{D}_{\alpha,\beta}^{(k)}(\id+\mathbb{L}_{\alpha,\beta})^{-k/2},&  \alpha+\beta=-1.
	\end{cases}
\end{gather*}
Notice that $\mathfrak{R}_{\alpha,\beta}^{k}$ is well def\/ined on $\mathbb{S}_{\alpha,\beta}$. In the structural theorem below
both $\mathfrak{R}_{\alpha,\beta}^{k}$ and $\mathfrak{D}_{\alpha,\beta}^{(k)}$ are understood as operators given initially
on $\mathbb{S}_{\alpha,\beta}$.
\begin{thm} \label{thm:struc}
Let ${\alpha,\beta} > -1$ and $p,q \in E({\alpha,\beta})$. Assume that $s,t>0$ and $k>0$ is even.
\begin{itemize}\itemsep=0pt
\item[$(i)$] If $p < q$, then $\mathcal{L}_{\alpha,\beta}^{q,s}(-\pi,\pi) \subset \mathcal{L}_{\alpha,\beta}^{p,s}(-\pi,\pi)$.
\item[$(ii)$] If $s < t$, then
$\mathcal{L}_{\alpha,\beta}^{p,t}(-\pi,\pi) \subset \mathcal{L}_{\alpha,\beta}^{p,s}(-\pi,\pi) \subset L^p(-\pi,\pi)$.
\end{itemize}
Moreover, the embeddings in $(i)$ and $(ii)$ are proper and continuous.
\begin{itemize}\itemsep=0pt
\item[$(iii)$] $\mathbb{L}_{\alpha,\beta}^{-t/2}$ establishes an isometric isomorphism between
 $\mathcal{L}_{\alpha,\beta}^{p,s}(-\pi,\pi)$ and $\mathcal{L}_{\alpha,\beta}^{p,s+t}(-\pi,\pi)$.
\item[$(iv)$] If $k<s$, then $\mathfrak{D}_{\alpha,\beta}^{(k)}$ extends to a bounded operator from
	$\mathcal{L}_{\alpha,\beta}^{p,s}(-\pi,\pi)$ to $\mathcal{L}_{\alpha+k,\beta+k}^{p,s-k}(-\pi,\pi)$.
	Moreover, $\mathfrak{D}_{\alpha,\beta}^{(k)}$ extends to a bounded operator from
	$\mathcal{L}_{\alpha,\beta}^{p,k}(-\pi,\pi)$ to $L^p(-\pi,\pi)$.
\item[$(v)$] $\mathfrak{R}_{\alpha,\beta}^k$ extends to a bounded operator from 	
	$\mathcal{L}_{\alpha,\beta}^{p,s}(-\pi,\pi)$ to $\mathcal{L}_{\alpha+k,\beta+k}^{p,s}(-\pi,\pi)$.
\end{itemize}
\end{thm}

\begin{proof}
Throughout the proof we assume $\alpha+\beta \neq -1$. The opposite case requires analogous arguments (with~(iv) and~(v) requiring a little bit more attention) and is left
to the reader.

To justify (i) it suf\/f\/ices to observe that $\|\cdot\|_{L^p(-\pi,\pi)}$ is controlled by
$\|\cdot\|_{L^q(-\pi,\pi)}$ whenever $p < q$. To demonstrate~(ii), let
$f\in\mathcal{L}_{\alpha,\beta}^{p,t}(-\pi,\pi)$. Then there exists $g\in L^p(-\pi,\pi)$ such that
$f=\mathbb{L}_{\alpha,\beta}^{-t/2}g$, so by Proposition~\ref{ogr}
\begin{gather}\label{comp}
f=\mathbb{L}_{\alpha,\beta}^{-t/2}g=\mathbb{L}_{\alpha,\beta}^{-s/2}\mathbb{L}_{\alpha,\beta}^{-(t-s)/2}g
	\in\mathcal{L}_{\alpha,\beta}^{p,s}(-\pi,\pi).
\end{gather}
Here second identity is straightforward for $g\in \mathbb{S}_{\alpha,\beta}$, and for $g \in L^p(-\pi,\pi)$
it follows by Proposition~\ref{ogr} and an approximation argument.

In view of the above, the inclusions in (i) and (ii) are continuous.
To see that the inclusion in~(i) is proper it suf\/f\/ices to take $f=\mathbb{L}_{\alpha,\beta}^{-s/2}g$ with some
$g\in L^p(-\pi,\pi){\setminus} L^q(-\pi,\pi)$. Then obviously $f\in \mathcal{L}_{\alpha,\beta}^{p,s}(-\pi,\pi)$.
On the other hand, $f\notin \mathcal{L}_{\alpha,\beta}^{q,s}(-\pi,\pi)$. Indeed, otherwise there would exist
$\widetilde{g}\in L^q(-\pi,\pi)$ such that $f=\mathbb{L}_{\alpha,\beta}^{-s/2}\widetilde{g}$, a contradiction with
injectivity of $\mathbb{L}_{\alpha,\beta}^{-s/2}$ on $L^p(-\pi,\pi)$, see Proposition~\ref{inj}.
To prove that the inclusions in (ii) are proper we f\/irst recall that an analogous result holds in the
non-symmetrized setting, see \cite[Theorem~3.1(i)]{L3}. This means that there exist
$h_1\in \mathcal{L}_{\alpha,\beta}^{p,s}(0,\pi){\setminus} \mathcal{L}_{\alpha,\beta}^{p,t}(0,\pi)$
and $h_2 \in L^p(0,\pi){\setminus} \mathcal{L}_{\alpha,\beta}^{p,s}(0,\pi)$.
Let $f_1$ and $f_2$ be even extensions of~$h_1$ and~$h_2$, respectively, to~$(-\pi,\pi)$.
Then Proposition~\ref{pot} implies
$f_1\in\mathcal{L}_{\alpha,\beta}^{p,s}(-\pi,\pi){\setminus}\mathcal{L}_{\alpha,\beta}^{p,t}(-\pi,\pi)$
and $f_2\in L^p(-\pi,\pi){\setminus}\mathcal{L}_{\alpha,\beta}^{p,s}(-\pi,\pi)$.

To show that (iii) holds, it suf\/f\/ices to observe that, see~\eqref{comp},
\begin{gather*}
\mathbb{L}_{\alpha,\beta}^{-t/2} \mathbb{L}_{\alpha,\beta}^{-s/2}g = \mathbb{L}_{\alpha,\beta}^{-(t+s)/2}g, \qquad g \in L^p(-\pi,\pi).
\end{gather*}

Proving (iv) is reduced to showing the bound
\begin{gather*}
\big\|\mathbb{L}_{\alpha+k,\beta+k}^{(s-k)/2}\mathfrak{D}_{\alpha,\beta}^{(k)}
	\mathbb{L}_{\alpha,\beta}^{-s/2}g\big\|_{L^p(-\pi,\pi)}\lesssim\|g\|_{L^p(-\pi,\pi)}, \qquad g \in \mathbb{S}_{\alpha,\beta}.
\end{gather*}
Since $k$ is even,
$\mathfrak{d}_{\alpha,\beta}^{(k)}g_{\textrm{even}}$ and $\mathfrak{d}_{\alpha+1,\beta+1}^{(k)}g_{\textrm{odd}}$
are even and odd functions, respectively. Therefore, in view of~\eqref{dec},
\begin{gather*}
\mathbb{L}_{\alpha+k,\beta+k}^{(s-k)/2}\mathfrak{D}_{\alpha,\beta}^{(k)}\mathbb{L}_{\alpha,\beta}^{-s/2}g
=\big(\mathbb{L}_{\alpha+k,\beta+k}^{(s-k)/2}\big)_{\textrm{e}}\mathfrak{d}_{\alpha,\beta}^{(k)}
\big(\mathbb{L}_{\alpha,\beta}^{-s/2}\big)_{\textrm{e}} g_{\textrm{even}}\!
+\!\big(\mathbb{L}_{\alpha+k,\beta+k}^{(s-k)/2}\big)_{\textrm{o}} \mathfrak{d}_{\alpha+1,\beta+1}^{(k)}
\big(\mathbb{L}_{\alpha, \beta}^{-s/2}\big)_{\textrm{o}} g_{\textrm{odd}}.
\end{gather*}
Now recall that in the non-symmetrized setting, see \cite[Theorem~3.1(iii)]{L3}, we have
\begin{gather*}
\big\|L_{\alpha+k,\beta+k}^{(s-k)/2}\mathfrak{d}_{\alpha,\beta}^{(k)}L_{\alpha,\beta}^{-s/2}h\big\|_{L^p(0,\pi)}
\lesssim\|h\|_{L^p(0,\pi)}, \qquad h \in S_{\alpha,\beta}.
\end{gather*}
Using this bound and its variant with $\alpha$ and $\beta$ replaced by $\alpha+1$ and $\beta+1$, respectively,
we obtain
\begin{gather*}
\big\|\mathbb{L}_{\alpha+k, \beta+k}^{(s-k)/2}\mathfrak{D}_{\alpha,\beta}^{(k)}
	\mathbb{L}_{\alpha,\beta}^{-s/2}g\big\|_{L^p(-\pi,\pi)}\\
\simeq \big\|\big(\mathbb{L}_{\alpha+k, \beta+k}^{(s-k)/2}\big)_{\textrm{e}}\mathfrak{d}_{\alpha,\beta}^{(k)}
	\big(\mathbb{L}_{\alpha,\beta}^{-s/2}\big)_{\textrm{e}} g_{\textrm{even}}\big\|_{L^p(-\pi,\pi)}
+\big\|\big(\mathbb{L}_{\alpha+k, \beta+k}^{(s-k)/2}\big)_{\textrm{o}}\mathfrak{d}_{\alpha+1,\beta+1}^{(k)}
	\big(\mathbb{L}_{\alpha, \beta}^{-s/2}\big)_{\textrm{o}} g_{\textrm{odd}}\big\|_{L^p(-\pi,\pi)}\\
\simeq\big\|L_{\alpha+k, \beta+k}^{(s-k)/2}\mathfrak{d}_{\alpha,\beta}^{(k)}L_{\alpha,\beta}^{-s/2}
	g_{\textrm{even}}^+\big\|_{L^p(0,\pi)}
+\big\|L_{\alpha+1+k, \beta+1+k}^{(s-k)/2}\mathfrak{d}_{\alpha+1,\beta+1}^{(k)}
	L_{\alpha+1, \beta+1}^{-s/2}g_{\textrm{odd}}^+\big\|_{L^p(0,\pi)}\\
\lesssim \big\|g_{\textrm{even}}^+\big\|_{L^p(0,\pi)}
	+\big\|g_{\textrm{odd}}^+\big\|_{L^p(0,\pi)}\simeq\|g\|_{L^p(-\pi,\pi)},
\end{gather*}
as required.

Finally, (v) is a consequence of (iv) and the boundedness of $\mathbb{L}_{\alpha,\beta}^{-k/2}$ from
$\mathcal{L}_{\alpha,\beta}^{p,s}(-\pi,\pi)$ to $\mathcal{L}_{\alpha,\beta}^{p,s+k}(-\pi,\pi)$, see (iii).
\end{proof}

Parts (iv) and (v) of Theorem \ref{thm:struc}, as stated, do not hold when $k$ is an odd number. Roughly speaking,
this is because in such a case $\mathfrak{D}_{\alpha,\beta}^{(k)}$ switches symmetry of functions from
even to odd and vice versa. However, a slight modif\/ication of $\mathfrak{D}_{\alpha,\beta}^{(k)}$ makes
the statements (iv) and (v) true for all $k \ge 1$. Indeed, one easily verif\/ies that the arguments
proving (iv) and (v) go through with $\mathfrak{D}_{\alpha,\beta}^{(k)}$ replaced by
$\widetilde{\mathfrak{D}}_{\alpha,\beta}^{(k)}$ def\/ined by
$\widetilde{\mathfrak{D}}_{\alpha,\beta}^{(k)}f(\theta)= \sign^k(\theta) \mathfrak{D}_{\alpha,\beta}^{(k)}f(\theta)$.

Another interesting question is whether there are any inclusions between potential spaces with
dif\/ferent parameters of type. It turns out that in general the answer is negative.
\begin{prop}
Let $\alpha, \beta, \gamma, \delta>-1$ be such that $(\alpha,\beta)\neq (\gamma, \delta)$.
Assume that $p \in E({\alpha,\beta}) \cap E(\gamma,\delta)$ and $\alpha,\beta,\gamma,\delta \le -1/p+1/2$.
Then neither the inclusion $\mathcal{L}_{\alpha,\beta}^{p,1}(-\pi,\pi)\subset
\mathcal{L}_{\gamma, \delta}^{p,1}(-\pi,\pi)$ nor the inclusion
$\mathcal{L}_{\gamma, \delta}^{p,1}(-\pi,\pi)\subset \mathcal{L}_{\alpha,\beta}^{p,1}(-\pi,\pi)$ holds.
\end{prop}

\begin{proof}
Take $f_1 = \Psi^{\alpha,\beta}$ and $f_2 = \Psi^{\gamma,\delta}$.
With the aid of \cite[formula~(8)]{L2} one verif\/ies that $f_1^+, f_2^+$, $D_{\alpha,\beta}f_1^+,
D_{\gamma,\delta}f_2^+\in L^p(0,\pi)$, but $D_{\alpha,\beta}f_2^+, D_{\gamma,\delta}f_1^+ \notin L^p(0,\pi)$.
Thus $f_1^+\in W_{\alpha,\beta}^{p,1}(0,\pi){\setminus} W_{\gamma,\delta}^{p,1}(0,\pi)$ and
$f_2^+\in W_{\gamma,\delta}^{p,1}(0,\pi){\setminus} W_{\alpha,\beta}^{p,1}(0,\pi)$.
Since~$f_1$ and~$f_2$ are even functions, it follows by Proposition~\ref{sob} that
$f_1\in\mathcal{L}_{\alpha,\beta}^{p,1}(-\pi,\pi){\setminus} \mathcal{L}_{\gamma,\delta}^{p,1}(-\pi,\pi)$ and
$f_2\in\mathcal{L}_{\gamma,\delta}^{p,1}(-\pi,\pi){\setminus} \mathcal{L}_{\alpha,\beta}^{p,1}(-\pi,\pi)$.
\end{proof}

We f\/inish this section with a counterpart of the classical Sobolev embedding theorem.
The statement below is a direct consequence of an analogous result
in the non-symmetrized situation \cite[Theorem~3.2]{L3} and Proposition~\ref{pot}.
We leave the details to the interested readers.

\begin{thm} \label{thm:embed}
Let ${\alpha,\beta} > -1$, $p \in E({\alpha,\beta})$ and $1\le q < p({\alpha,\beta})$.
\begin{itemize}\itemsep=0pt
\item[$(i)$]
If $s > 0$ is such that $1/q \ge 1/p -s$, then $\mathcal{L}_{\alpha,\beta}^{p,s}(-\pi,\pi) \subset L^q(-\pi,\pi)$ and
\begin{gather} \label{embest}
\|f\|_{q} \lesssim \|f\|_{\mathcal{L}_{\alpha,\beta}^{p,s}(-\pi,\pi)}, \qquad f \in \mathcal{L}_{\alpha,\beta}^{p,s}(-\pi,\pi).
\end{gather}
\item[$(ii)$]
If ${\alpha,\beta} \ge -1/2$ and $s > 1/p$, then $\mathcal{L}_{\alpha,\beta}^{p,s}(-\pi,\pi) \subset C(-\pi,\pi)$
and \eqref{embest} holds with $q=\infty$.
\end{itemize}
\end{thm}

\section{Sample applications of potential spaces}

The study of symmetrized potential spaces performed in the previous sections reveals that the symmetrized
objects inherit many of the properties of their non-symmetrized prototypes. Furthermore, most of the proofs
were based on the arguments relying on suitable decompositions of the operators
into their even and odd symmetric parts. This actually reduced our problems to the non-symmetrized setup.
Therefore, it comes as no surprise that both theories, symmetrized and non-symmetrized, have parallel
applications. Below we present some results which illustrate the utility of the symmetrized potential spaces.
The proofs combine the symmetry arguments that have already appeared in this paper with the
results from \cite[Section 5]{L3}. We leave them to the reader.

Given some initial data $f \in L^2(-\pi,\pi)$,
consider the following Cauchy problem based on the symmetrized Jacobi operator:
\begin{gather*}
\begin{cases}
\big(i\partial_t + \mathbb{L}_{\alpha,\beta}\big) u(\theta,t)=0, \\
u(\theta,0)=f(\theta),
\end{cases}
\qquad  {\rm a.a.} \  \theta \in (-\pi,\pi),\quad t\in\mathbb{R}.
\end{gather*}
Then $\exp(it\mathbb{L}_{\alpha,\beta})f$, understood spectrally, is a solution to this problem.
The next result shows that the theory of symmetrized Jacobi potential spaces can be used to
study pointwise almost everywhere convergence of this solution to the initial condition.
\begin{prop}
Let ${\alpha,\beta} > -1$ and $s > 1/2$. Then for each $f\in \mathcal{L}_{\alpha,\beta}^{2,s}(-\pi,\pi)$
\begin{gather*}
\lim_{t\rightarrow0} \exp(it\mathbb{L}_{\alpha,\beta})f(\theta)=
	f(\theta) \qquad {\rm a.a.} \  \theta \in (-\pi,\pi).
\end{gather*}
\end{prop}

Furthermore, one can estimate a mixed norm of the solution in terms of the potential norm of the initial
condition.
\begin{prop}
Let ${\alpha,\beta} > -1$ and $p \in E({\alpha,\beta})$. Assume that $q \ge 2$ and $s>0$ is such that
$s \ge 3/2 + \max\{\alpha,\beta\}$ and $\alpha+\beta$ is integer. Then
\begin{gather*}
\big\|\exp(it\mathbb{L}_{\alpha,\beta})f\big\|_{L_{\theta}^p((-\pi,\pi), L_t^q(0,2\pi))} \lesssim
	\|f\|_{\mathcal L_{\alpha,\beta}^{2,s+1-2/q}(-\pi,\pi)},\qquad f\in \mathcal L_{\alpha,\beta}^{2,s+1-2/q}(-\pi,\pi).
\end{gather*}
\end{prop}

\subsection*{Acknowledgment}
The author would like to express his gratitude to Professor Adam Nowak for indicating the topic and
constant support during the preparation of this paper. 
Research supported by the National Science Centre of Poland, project No.~2013/09/N/ST1/04120.

\pdfbookmark[1]{References}{ref}
\LastPageEnding

\end{document}